\pdfoutput=1
\documentclass{wsbart/wsbart}

\newmuskip\tinymuskip
\renewcommand{\`}{\mskip-\tinymuskip}

\usepackage[utf8]{inputenc}
\usepackage[T1]{fontenc}
\usepackage{lmodern}

\usepackage[maxnames=99,style=trad-plain,backend=biber,block=space]{biblatex}
\newcommand{\citelink}[2]{\hyperlink{cite.\therefsection @#1}{#2}}
\DeclareRedundantLanguages{English}{english}

\usepackage{amsmath,amssymb}
\usepackage{amsthm}
\usepackage{bm,dsfont,stmaryrd}
\SetSymbolFont{stmry}{bold}{U}{stmry}{m}{n}
\renewcommand{\mathbb}{\mathds}
\usepackage{mathtools}
\usepackage{accents}
\usepackage{IEEEtrantools}

\newcommand{\R}{\mathbb R}

\newcommand{\vect}[1]{\bm{#1}}

\newcommand{\dd}{\mathop{}\!\mathrm{d}}
\DeclareMathOperator{\Law}{Law}
\DeclareMathOperator{\Proba}{\mathbb{P}}
\DeclareMathOperator{\Expect}{\mathbb{E}}
\newcommand{\1}{\mathbb{1}}

\newcommand{\adjustbin}{\negmedspace{}}
\newcommand{\HS}{\textnormal{HS}}

\newtheorem{thm}{Theorem}
\newtheorem{lem}[thm]{Lemma}
\newtheorem{cor}[thm]{Corollary}
\newtheorem{prop}[thm]{Proposition}
\theoremstyle{definition}

\newtheorem*{assu*}{Assumption}
\theoremstyle{remark}
\newtheorem{rem}{Remark}

\title{Uniform log-Sobolev inequalities for mean field particles
with flat-convex energy}
\author{Songbo Wang}
\affil{CMAP, École polytechnique, IP Paris, Palaiseau, France}

\begin{document}


\maketitle

\begin{abstract}
The purpose of this short note is to demonstrate
uniform logarithmic Sobolev inequalities
for the mean field gradient particle systems
associated to an energy functional that is convex in the flat sense.
A defective log-Sobolev inequality was already established implicitly
in a previous joint work with F.~Chen and Z.~Ren
[\citelink{ulpoc}{\texttt{arXiv:2212.03050 [math.PR]}}].
It remains only to tighten it by a uniform Poincaré inequality,
which we prove by the method
in a recent work of Guillin, W.~Liu, L.~Wu and C.~Zhang
[\citelink{GLWZUPLSI}{\textit{Ann.\ Appl.\ Probab.}, 32(3):1590–1614, 2022}].
As an application, we show that the particle system
exhibits the concentration of measure phenomenon in the long time.
\end{abstract}

\section{Introduction and main result}

Let $d$ be an integer $\geqslant 1$.
In this note, we consider convex \emph{mean field energy functionals}
$F$ that are defined on the space of probability measures on $\R^d$.
More precisely, such functionals are mappings
\[
F \colon \mathcal P_2(\R^d) \to \R,
\]
where $\mathcal P_2(\R^d)$ denotes the set of probability measures
of finite second moment;
and along flat interpolations in $\mathcal P_2(\R^d)$,
the energy goes through a convex trajectory.
Let $N$ be an integer $\geqslant 1$
and let $\vect x = (x^1, \ldots, x^N)$ be an $N$-tuple of coordinates in $\R^d$,
that is to say, $\vect x \in \R^{Nd}$.
We denote by $\mu_{\vect x}$ the empirical measure formed with the
$N$ coordinates $x^1$, \dots, $x^N$, that is,
\[
\mu_{\vect x} \coloneqq \frac 1N \sum_{i = 1}^N \delta_{x^i}.
\]
We are interested in proving that the \emph{$N$-particle Gibbs measure} $m^N_*$
for the energy $F$, defined by
\begin{equation}
\label{eq:def-ps-gibbs}
m^N_*(\dd\vect x) \coloneqq
\frac{\exp \bigl( - NF(\mu_{\vect x}) \bigr) \dd\vect x}
{\int_{\R^{Nd}}\exp \bigl( - NF(\mu_{\vect x'}) \bigr) \dd\vect x'},
\end{equation}
satisfies a \emph{logarithmic Sobolev inequality} uniformly in $N$:
in other words, for some sequence $(\rho^N)_{N \in \mathbb N}$ with
$\liminf_{N \to \infty} \rho^N > 0$, we have
\[
2 \rho \int_{\R^{Nd}} \log \frac{\dd m^N}{\dd m^N_*} \dd m^N
\eqqcolon 2 \tilde\rho H(m^N | m^N_*)
\leqslant I(m^N | m^N_*)
\coloneqq
\int_{\R^{Nd}} \biggl| \nabla \log \frac{\dd m^N}{\dd m^N_*} \biggr|^2
\dd m^N
\]
for all probability measures $m^N$ on $\R^{Nd}$
such that the relative density $\dd m^N\!/\! \dd m^N_*$
is $\mathcal C^1_\textnormal{b}$, that is,
bounded and $\mathcal C^1$ with bounded derivatives.
In the inequality above, we call the quantities $H$, $I$
the relative \emph{entropy} and \emph{Fisher information} respectively.
We also call $\rho^N$ a log-Sobolev constant for $m^N_*$
if the assertion above holds.

The main motivation behind our search for the log-Sobolev inequality
for $m^N_*$ is that it allows us to show
the \emph{exponential convergence in entropy}
for the system of diffusive particles:
\begin{equation}
\label{eq:ps}
\dd X^i_t = - D_m F(\mu_{\vect X_t}, X^i_t) \dd t + \sqrt 2 \dd B^i_t,\qquad
\text{for $i=1$, \dots, $N$.}
\end{equation}
Here $D_m F$ is called intrinsic derivative of $F$ and is a mapping
from $\mathcal P_2(\R^d) \times \R^d$ into $\R^d$
and will be defined in the following,
and $B^i_t$ are i.i.d.\ standard Brownian motions in $d$ dimensions.
The system of particles corresponds to a Liouville or Fokker--Planck equation
defined on $[0,\infty) \times \R^{Nd}$ for the flow of probability measures
\[
m^N_t = \Law(X^1_t, \ldots, X^N_t),\qquad\text{for $t \in [0,\infty)$}
\]
and by explicit computations, we can find that the $m^N_*$ is invariant
to the dynamics.
Moreover, if $m^N_*$ verifies a $\rho$-log-Sobolev inequality,
then for all initial value $m^N_0$ of the flow,
\[
H(m^N_t | m^N_*) \leqslant e^{-2\rho t} H(m^N_0 | m^N_*).
\]
See Section~5.2 of the Bakry--Gentil--Ledoux monograph \cite{BGLMarkov}
for details.
Thus the existence of a log-Sobolev constant independent from $N$
implies that the rate of convergence is independent of the number of particles.
We refer readers to the previous work joint with F.~Chen and Z.~Ren
\cite{ulpoc} for motivations behind the particle system,
which include in particular the modeling of shallow neural networks.

Apart from the log-Sobolev inequality,
the Poincaré inequality also plays a central role in the ergodic behavior
of diffusion processes.
For our probability measure of interest $m^N_*$,
we say that it satisfies a $\rho^N$-Poincaré inequality if
for all $f \in \mathcal C^1_\textnormal{b} (\R^{Nd}; \R)$,
\begin{equation}
\label{eq:ps-gibbs-pi}
\rho^N \biggl( \int_{\R^{Nd}} f^2 \dd m^N_*
- \Bigl( \int_{\R^{Nd}} f \dd m^N_* \Bigr)^{\!2} \biggr)
\leqslant \int_{\R^{Nd}} \lvert \nabla f \rvert^2 \dd m^N_*.
\end{equation}
The Poincaré implies equally exponential convergence for the diffusion process,
except that the relative entropy $H(m^N_t|m^N_*)$ must be replaced
by the $\chi^2$ divergence between $m^N_t$ and $m^N_*$,
which is a weighted $L^2$ distance.
We refer readers to Section~4.2 of \cite{BGLMarkov} for details.
Here we only mention that in this note,
we first obtain a uniform Poincaré inequality for $m^N_*$
by the method of Guillin, W.~Liu, L.~Wu and C.~Zhang \cite{GLWZUPLSI}.
Then we derive the stronger log-Sobolev inequality by
a ``tightening'' procedure.

\bigskip

We impose the following assumption on the energy $F$.

\begin{assu*}
The functional $F$ admits first and second-order flat derivatives
\[
\frac{\delta F}{\delta m} \colon \mathcal P_2(\R^d) \times \R^d \to \R,
\quad
\frac{\delta^2\`F}{\delta m^2} \colon
\mathcal P_2(\R^d) \times \R^d \times \R^d \to \R
\]
that are joint continuous and are $\mathcal C^2$ in the spatial variables.
(See \cite[Chapter 5]{CarmonaDelarueMFG1} for related definitions.)
Moreover, denoting
\[
D_m F(m,x) \coloneqq \nabla_x \frac{\delta F}{\delta m}(m,x),
\quad
D_m^2 F(m,x,x') \coloneqq
\nabla^2_{x,x'} \frac{\delta^2\`F}{\delta m^2}(m,x,x'),
\]
we have the following:
\begin{enumerate}
\item there exists an $M^F_{mm} \geqslant 0$
such that for all $m \in \mathcal P_2(\R^d)$ and $x$, $x' \in \R^d$,
the Euclidean operator norm of the matrix
$D_m^2F(m,x,x')$ does not exceed $M^F_{mm}$;
\item there exists an $M^F_{mx} \geqslant 0$
such that for all $m \in \mathcal P_2(\R^d)$ and $x \in \R^d$,
the Euclidean operator norm of the matrix
$\nabla_x D_m F(m,x)$ does not exceed $M^F_{mx}$;
\item there exists a $\rho > 0$ such that the probability measures $\hat m$
defined by
\[
\hat m(\dd x) \coloneqq
\frac{\exp \bigl( - \frac{\delta F}{\delta m}(m,x) \bigr) \dd x}
{\int_{\R^d}\exp \bigl( - \frac{\delta F}{\delta m}(m,x') \bigr) \dd x'}
\]
satisfy a $\rho$-log-Sobolev inequality uniformly
for $m \in \mathcal P_2(\R^d)$;
\item for the same $\rho$, the conditional distribution $m_*^{N,1|-1}$
with density
\[
m_*^{N,1|-1}(x^1 | \vect x^{-1})
\coloneqq \frac{\exp \bigl( - NF(\mu_{\vect x}) \bigr)}
{\int_{\R^d} \exp \bigl( - NF(\mu_{\vect x}) \bigr) \dd x^1}
\]
satisfies a $\rho$-Poincaré inequality
uniformly for $\vect x^{-1} \coloneqq (x^2,\ldots,x^N) \in \R^{(N-1)d}$;
\item the energy is convex in the \emph{flat interpolation} sense:
for all $m$, $m' \in \mathcal P_2(\R^d)$ and $\lambda \in [0,1]$,
\[
F\bigl( (1 - \lambda) m + \lambda m' \bigr)
\leqslant (1 - \lambda) F(m) + \lambda F(m').
\]
\end{enumerate}
\end{assu*}

The main result of this note is formulated as follows.

\begin{thm}
\label{thm:lsi}
If the energy functional $F$ satisfies the assumption above,
then for $N > M^F_{mm} / \rho \eqqcolon \alpha$,
its $N$-particle Gibbs measure $m^N_*$, defined by \eqref{eq:def-ps-gibbs},
satisfies a log-Sobolev inequality
with the constant
\[
\rho^N \coloneqq
\frac{1 - \varepsilon - \bigl(8\alpha+6(\varepsilon^{-1}-1)\bigr)
\frac{\alpha^2}{N}}
{1 + 2d \bigl(5 + 3(\varepsilon^{-1} - 1) \alpha\bigr)
\frac{\alpha}{1 - \alpha/N}} \rho,
\]
where $\varepsilon \in (0,1)$ is arbitrary.
\end{thm}

\begin{rem}
Note that in the expression for $\rho^N$ above,
fixing the value of $\varepsilon$, we have
\[
\lim_{N \to \infty} \rho^N
= \frac{1 - \varepsilon} {1 + 2d\alpha
\bigl(5 + 3(\varepsilon^{-1} - 1) \alpha\bigr)}
\rho > 0.
\]
So the result gives indeed a log-Sobolev inequality uniform in $N$.
However, unless $\alpha = 0$, we cannot find $\varepsilon$
such that
\[
\lim_{N \to \infty} \rho^N = \rho,
\]
as conjectured by Delgadino, Gvalani, Pavliotis and Smith
in \cite{DGPSPhase}.
Moreover, the constant $\rho^N$ obtained by our method
becomes weaker when the dimension $d$ increases, making it possibly
unsuitable for applications in high dimensions.
The author does not know if these behaviors can be avoided.
\end{rem}

Then we show that the log-Sobolev inequality for $m^N_*$
implies the concentration of measure phenomenon
for the particle system \eqref{eq:ps} in the long time.
To achieve this, we first demonstrate a general result
for symmetric diffusions on $\R^d$.

\begin{thm}
\label{thm:concentration}
Let $m_*$ be a probability measure on $\R^d$
which admits the density $m_*(x) = \exp\bigl( - U(x) \bigr)$
for some $U \in \mathcal C^2(\R^d; \R)$.
Suppose that the Euclidean operator norm of the Hessian $\nabla^2 U$
is bounded by some $M > 0$
and $m_*$ satisfies a $\rho$-log-Sobolev inequality for some $\rho > 0$.
Let $(X_t)_{t \geqslant 0}$ be the overdamped Langevin particle:
\[
\dd X_t = - \nabla U(X_t) \dd t + \sqrt 2 \dd B_t,
\]
where $B_t$ is a standard Brownian motion in $\R^d$.
Then, for all $t \geqslant 1$ and all $1$-Lipschitz $f \colon \R^d \to \R$,
\begin{multline*}
\Proba\bigl[f(X_t) - m_*[f] \geqslant r\bigr] \\
\leqslant \int_{\R^d}
\exp \biggl( \frac{M^2+3}{6} e^{-\rho(t-1)} W_2^2(\delta_x, m_*)\biggr)
m_0(\dd x) \exp \biggl( - \frac{\rho r^2}{4} \biggr),
\end{multline*}
where $m_*[f]$ denotes the integral $\int_{\R^d} f\dd m_*$.
\end{thm}

\begin{rem}
Denote $m_t = \Law(X_t)$ for $t \geqslant 0$.
In other words, if the initial distribution $m_0$ has a finite Gaussian moment,
that is, there exists $\varepsilon > 0$ such that
$\int \exp (\varepsilon \lvert x\rvert^2) m_0(\dd x) < \infty$,
then $m_t$ has a uniform Gaussian tail for sufficiently large $t$.
The Gaussian integrability is indeed necessary.
Take the example of Ornstein--Uhlenbeck semigroup where
$U(x) = \lvert x\rvert^2\!/2$.
By explicit computations, we can prove that for all $t > 0$,
$m_t$ has a finite Gaussian moment if and only if
$m_0$ has a finite Gaussian moment.
\end{rem}

Note that our diffusion process \eqref{eq:ps} has the Langevin potential
\[
U^N(\vect x) \coloneqq N F(\mu_{\vect x})
\]
with second-order derivatives
\[
\nabla^2_{i,j}U^N(\vect x) = \nabla D_m F(\mu_{\vect x}, x^i) \1_{i = j}
+ \frac 1N D_m^2F (\mu_{\vect x}, x^i, x^j).
\]
We then apply Theorem~\ref{thm:concentration}
by plugging $m_* \to m^N_*$, $U \to U^N$ with constants
$M \to M^F_{mm} + M^F_{mx}$, $\rho \to \rho^N$
and with the $1/\sqrt N$-Lipschitz test function
\[
\R^{Nd} \ni \vect x = (x^1, \ldots, x^N)
\mapsto \frac{f(x^1) + \cdots + f(x^N)}{N} \in \R.
\]
Immediately we get the following result.

\begin{cor}
Under the setting of Theorem~\ref{thm:lsi},
for all $t \geqslant 1$, $r \geqslant 0$ and
all $1$-Lipschitz function $f \colon \R^d \to \R$,
\begin{multline*}
\Proba \biggl[ \frac{f(X^1_t) + \cdots + f(X^N_t)}{N}
- \Expect [ f(X_*) ] \geqslant r\biggr] \\
\leqslant \int_{\R^{Nd}}
\exp \biggl( \frac{(M^F_{mm} + M^F_{mx})^2+3}{6}
e^{-\rho^N(t-1)} W_2^2(\delta_{\vect x}, m^N_*)\biggr)
m^N_0(\dd\vect x) \\
\exp \biggl( - \frac{N\rho^N r^2}{4} \biggr),
\end{multline*}
where $X_*$ is distributed as the $1$-marginal of $m^N_*$.
\end{cor}

We give the proofs of the theorems in the following two sections respectively
and make some additional comments in the end.

\section{Proof of Theorem~\ref{thm:lsi}}

The proof consists of three steps.

\proofstep{Step 1: Defective log-Sobolev inequality}
This step has essentially been established in the previous work \cite{ulpoc}.
In the end of the proof of Theorem 1.12 in that article,
we established the following functional inequality:
for all probability measure $m^N$ on $\R^{Nd}$
such that $\dd m^N\!/\!\dd m^N_*$ is $\mathcal C^1_\textnormal{b}$,
and for all $\varepsilon \in (0,1)$,
\begin{IEEEeqnarray*}{rCl}
I(m^N | m^N_*)
&\geqslant & 2\biggl( (1-\varepsilon) \rho
- \frac{M^F_{mm}}{N}
\Bigl( 8 + 6 (\varepsilon^{-1} - 1) \frac{M^F_{mm}}{\rho} \Bigr) \biggr)
\bigl( \mathcal F^N(m^N) - N \mathcal F(m_*) \bigr) \\
&& \quad\adjustbin - 2d M^F_{mm} \Bigl( 5 + 3 (\varepsilon^{-1} - 1)
\frac{M^F_{mm}}{\rho} \Bigr)\\
&\eqqcolon& 2 \rho'
\bigl( \mathcal F^N(m^N) - N \mathcal F(m_*) \bigr)
- \delta,
\end{IEEEeqnarray*}
where $\mathcal F^N$ and $\mathcal F$ denote respectively
\begin{align*}
\mathcal F^N(m^N) &\coloneqq N \int_{\R^{Nd}} F(\mu_{\vect x}) m^N(\dd\vect x)
+ H(m^N), \\
\mathcal F(m) &\coloneqq F(m) + H(m),
\end{align*}
and $m_*$ is the unique probability measure on $\mathbb R^d$
that reaches the minimum of $\mathcal F$.
See Section~4 of \cite{ulpoc} for details and note that
the definition for the log-Sobolev constant there differs by a factor of $2$.
Again, Lemma~5.2 in the article gives that
for all $m^N \in \mathcal P_2(\mathbb R^{Nd})$ with finite entropy,
\[
\mathcal F^N(m^N) - N\mathcal F(m_*) \geqslant H(m^N | m_*^{\otimes N}).
\]
Taking $m^N = m^N_*$ yields
\[
\mathcal F^N(m^N_*) - N\mathcal F(m_*) \geqslant H(m^N_* | m_*^{\otimes N})
\geqslant 0.
\]
Thus,
\[
\mathcal F^N(m^N) - N\mathcal F(m_*)
\geqslant \mathcal F^N(m^N) - \mathcal F^N(m^N_*) = H(m^N | m^N_*).
\]
Combining the functional inequality and the inequality above,
we deduce that
\[
I(m^N | m^N_*) \geqslant 2 \rho' H(m^N | m^N_*) - \delta,
\]
which is a defective log-Sobolev inequality for the Gibbs measure $m^N_*$.
In the following we say that a measure satisfies a $(\rho',\delta)$%
-defective log-Sobolev inequality if the inequality above
is satisfied when $m^N_*$ is replaced by that measure.
To recover the usual form of defective log-Sobolev,
we denote $\phi (t) = t \log t$ for $t \geqslant 0$
and introduce the new variable $f = \sqrt{\dd m^N\!/\!\dd m^N_*}$.
As $\rho' > 0$, the defective log-Sobolev then reads:
for all $f \in \mathcal C^1_\textnormal{b}$,
\begin{equation}
\label{eq:ps-gibbs-dlsi}
\int_{\R^{Nd}} \phi(f^2) \dd m^N_*
- \phi \biggl( \int_{\R^{Nd}} f^2 \dd m^N_* \biggr)
\leqslant \frac{2}{\rho'}
\int_{\R^{Nd}} \lvert \nabla f\rvert^2 \dd m^N_*
+ \frac{\delta}{2\rho'} \int_{\R^{Nd}} f^2 \dd m^N_*,
\end{equation}
which is in line with Definition~5.1.1 of \cite{BGLMarkov}.

\proofstep{Step 2: Poincaré inequality}
In this step, we follow the approach in \cite{GLWZUPLSI},
especially that of Example~2 therein,
to prove a uniform Poincaré inequality for $m^N_*$.
Recall that the measure $m^N_*$ corresponds to the Langevin potential
\[
U^N(\vect x) = NF(\mu_{\vect x})
\]
and denote
\[
\mathcal L^N \coloneqq \Delta - \nabla U^N \cdot \nabla
= \Delta - \sum_{i=1}^N D_m F(\mu_{\vect x}, x^i) \cdot \nabla_i.
\]
According to Proposition~4.8.3 in \cite{BGLMarkov},
the Poincaré inequality \eqref{eq:ps-gibbs-pi} is equivalent to the following
``second-order'' inequality: for all $f \in \mathcal C^1_\textnormal{b}$,
\begin{equation}
\label{eq:ps-gibbs-second-pi}
\rho \int_{\R^{Nd}} \lvert \nabla f\rvert^2 \dd m^N_*
\leqslant \int_{\R^{Nd}} (\mathcal L^N\!f)^2 \dd m^N_*.
\end{equation}
We aim to prove this inequality in the following.

By the $\Gamma_{\`2}$ calculus
(see Proposition~3.3.16 of \cite{BGLMarkov}),
the right-hand side satisfies
\begin{IEEEeqnarray*}{rCl}
\int_{\R^{Nd}} (\mathcal L^N\!f)^2 \dd m^N_*
&=& \int_{\R^{Nd}} \Gamma_{\`2} (f) \dd m^N_* \\
&\coloneqq& \sum_{i,j=1}^N \int_{\R^{Nd}}
(\lvert \nabla^2_{i,j} f\rvert_\HS^2
+ \nabla_i f \nabla^2_{i,j} U^N \nabla_j f) \dd m^N_*,
\end{IEEEeqnarray*}
where $\lvert \cdot \rvert_\HS$ denotes the Hilbert--Schmidt norm of matrices.
It remains to lower bound the $\Gamma_{\`2}$ term on the right.
Note that the Hessian of $U^N$ reads
\[
\nabla^2_{i,j} U^N (\vect x)
= \frac 1N D_m^2 F(\mu_{\vect x}, x^i, x^j)
+ \nabla D_m F(\mu_{\vect x}, x^i) \1_{i = j},
\]
where $\1_{\cdot}$ is the indicator function.
Thus,
\begin{IEEEeqnarray*}{rCl}
\Gamma_{\`2}(f)
&\geqslant& \sum_{i=1}^N \,\lvert \nabla^2_i f \rvert^2_\HS
+ \nabla_i f \nabla D_m F(\mu_{\vect x}, x^i) \nabla_i f \\
&& \quad\adjustbin
+ \frac 1N \sum_{i,j = 1}^N
\nabla_i f D_m^2 F(\mu_{\vect x}, x^i, x^j) \nabla_j f.
\end{IEEEeqnarray*}
To proceed, we need the following lemma on flat convexity.

\begin{lem}
\label{lem:convexity}
Let $W \colon \R^d \times \R^d \to \R$ be $\mathcal C^2$ continuous.
Suppose that $W$ is of positive type, that is,
for all signed measure $\mu$ on $\R^d$ with $\int_{\R^d} \dd \mu = 0$,
\[\iint_{\R^d \times \R^d} W(x,x') \mu^{\otimes 2}(\dd x\dd x') \geqslant 0.\]
Then, for all integer $N \geqslant 1$,
and $v^i$, $x^i \in \R^d$ for $i = 1$, \dots, $N$,
\[
\sum_{i,j=1}^N v^i \nabla^2_{1,2} W(x^i, x^j) v^j \geqslant 0,
\]
where $\nabla^2_{1,2}$ means the composition
of the partial differential operators
with respect to the first and the second variable.
\end{lem}

\begin{proof}[Proof of Lemma~\ref{lem:convexity}]
Let $v_i$, $x_i$ be as in the statement and let $h > 0$.
Form the empirical measure
\[
\mu_h \coloneqq \sum_{i=1}^N \delta_{x^i + h v^i} - \delta_{x^i}.
\]
Since $W$ is of positive type, we have
\[
\iint_{\R^d \times \R^d} W \dd \mu_h^{\otimes 2} \geqslant 0.
\]
To conclude, it suffices to note that by the $\mathcal C^2$ continuity of $W$,
\[
\sum_{i,j=1}^N v^i \nabla^2_{1,2} W(x^i, x^j) v^j
= \lim_{h \searrow 0} \frac 1{h^2}
\iint_{\R^d \times \R^d} W \dd \mu_h^{\otimes 2}. \qedhere
\]
\end{proof}

The flat convexity of $F$ implies that
for all $m \in \mathcal P_2(\R^d)$,
the second-order flat derivative
$\frac{\delta^2\`F}{\delta m^2}(m,\cdot,\cdot)$
is a function of positive type.
Thus, taking
$W = \frac{\delta^2\`F}{\delta m^2}(\mu_{\vect x}, \cdot, \cdot)$
and plugging in $v^i = \nabla_i f$, we get that for all $\vect x \in \R^{Nd}$,
\[
\sum_{i,j=1}^N
\nabla_i f(\vect x) D_m^2 F(\mu_{\vect x}, x^i, x^j) \nabla_j f(\vect x)
\geqslant 0.
\]
It follows that
\begin{IEEEeqnarray*}{rCl}
\IEEEeqnarraymulticol{3}{l}
{\int_{\R^{Nd}} \Gamma_{\`2}(f) \dd m^N_*} \\
\quad&\geqslant& \sum_{i=1}^N \int_{\R^{Nd}}
\bigl(\lvert \nabla_i f\rvert_\HS^2
+ \nabla_i f\nabla D_m F(\mu_{\vect x}, x^i) \nabla_i f\bigr)
m^N_*(\dd\vect x) \\
&\geqslant& \sum_{i=1}^N \iint_{\R^{Nd}}
\bigl(\lvert \nabla_i f\rvert_\HS^2
+ \nabla_i f\nabla D_m F(\mu_{\vect x}, x^i) \nabla_i f\bigr)
m^{N,i|-i}_* (\dd x^i | \vect x^{-i}) m^{N,-i}_* (\dd\vect x^{-i}),
\end{IEEEeqnarray*}
where $\vect x^{-i} \coloneqq (x^j)_{j \neq i}$
and $m^{N,i|-i}_{*}$ is the conditional distribution of the $i$-th particle
given the position of all other particles.
The density of the conditional measure $m^{N,i|-i}_*$ verifies
\[
- \nabla^2_i \log m^{N,i|-i}_* (x^i | \vect x^{-i})
= \nabla D_m F(\mu_{\vect x}, x^i) + \frac 1N D_m^2 F(\mu_{\vect x}, x^i, x^i),
\]
where an extra term involving $D_m^2F$ appears.
By our assumption, the conditional measure
satisfies a $\rho$-Poincaré inequality
and the equivalent form of the Poincaré shows that
for all $\vect x^{-i}$,
\begin{multline*}
\int_{\R^{d}}
\biggl(\lvert \nabla_i f\rvert_\HS^2
+ \nabla_i f \Bigl(
\nabla D_m F(\mu_{\vect x}, x^i) + \frac 1N D_m^2F(\mu_{\vect x},x^i,x^i)
\Bigr) \nabla_i f\biggr)
m^{N,i|-i}_* (\dd x^i | \vect x^{-i}) \\
\geqslant \rho \int_{\R^d} \lvert \nabla_i f\rvert^2
m^{N,i|-i}_*(\dd x^i|\vect x^{-i}).
\end{multline*}
According to our assumption, the extra term satisfies
\[
\lvert \nabla_i f D_m^2 F(\mu_{\vect x}, x^i, x^i) \nabla_i f\rvert
\leqslant M^F_{mm} \lvert \nabla_i f\rvert^2.
\]
Integrating against the marginal distribution $m^{N,-i}_*$ and summing over $i$,
we get
\begin{multline*}
\sum_{i,j=1}^N\int_{\R^{Nd}}
\bigl(\lvert \nabla_i f\rvert_\HS^2
+ \nabla_i f \nabla D_m F(\mu_{\vect x}, x^i) \nabla_i f\bigr)
m^{N}_* (\dd \vect x) \\
\geqslant \biggl(\rho - \frac{M^F_{mm}}{N}\biggr)
\sum_{i = 1}^N\int_{\R^{Nd}} \lvert \nabla_i f\rvert^2
m^{N}_* (\dd \vect x).
\end{multline*}
So by combining the lower bound on
$\int_{\R^{Nd}} \Gamma_{\`2}(f)\dd m^N_*$,
the desired second-order Poincaré inequality
\eqref{eq:ps-gibbs-second-pi} is proved with constant $\rho - M^F_{mm} / N$.
The original Poincaré inequality \eqref{eq:ps-gibbs-pi} follows
by the equivalence mentioned above.

\proofstep{Step 3: Tightening}
By the standard tightening method (see Proposition~5.1.3 of \cite{BGLMarkov}),
the defective log-Sobolev \eqref{eq:ps-gibbs-dlsi}
and the Poincaré \eqref{eq:ps-gibbs-pi}
imply a log-Sobolev inequality with constant
\[
\biggl(
\frac{1}{\rho'} + \frac{1}{\rho - M^F_{mm}/N}
\Bigl( \frac{\delta}{4\rho'} + 1\Bigr)
\biggr)^{\!-1}.
\]
However, the constant above does not tend to $\rho$
under the limit $\delta \to 0$ due to the constant $1$ term.
This term is
introduced by Rothaus's lemma (see Lemma 5.1.4 of \cite{BGLMarkov})
employed in the usual log-Sobolev tightening proof.
To avoid this behavior, we give a new proof of the tightening
based on hypercontractivity.

\begin{prop}
\label{prop:tightening}
Let $\mu$ be a probability measure on $\R^d$.
Suppose that it admits density $\mu(x) = \exp \bigl( - U(x)\bigr)$
for some $U \in \mathcal C^2(\R^d; \R)$ with bounded $\nabla^2U$.
Suppose that $\mu$ satisfies
a $(\rho_1, \delta)$-defective log-Sobolev inequality
and a $\rho_2$-Poincaré inequality.
Then $\mu$ satisfies a log-Sobolev inequality with constant
\[
\frac{\rho_1\rho_2}{\rho_2 + \delta/4}.
\]
\end{prop}

\begin{proof}[Proof of Proposition~\ref{prop:tightening}]
Denote
\begin{align*}
\mathcal L &\coloneqq \Delta - \nabla U \cdot \nabla, \\
P_t &\coloneqq \exp( t\mathcal L).
\end{align*}
It is well known that by the argument of Gross
(see Theorem~5.2.3 of \cite{BGLMarkov}),
the $(\rho_1,\delta)$-defective log-Sobolev
implies the hyperboundness of the diffusion semi-group $P_t$:
for all $t_1 \in (0,1]$,
\[
\lVert P_{t_1} \rVert_{2 \to q(t_1)} \leqslant \exp \bigl( M(t_1) \bigr)
\]
where on the left $\lVert\cdot\rVert_{2\to q(t_1)}$ denotes the operator norm
from $L^2(\mu)$ into $L^{q(t_1)}(\mu)$
and $q$, $M$ are defined by
\begin{align*}
q(t_1) &\coloneqq 1 + \exp(2\rho_1 t_1) = 2 + 2\rho_1 t_1 + O(t_1^2), \\
M(t_1) &\coloneqq \frac{\delta}{2\rho_1}
\biggl( \frac 12 - \frac 1{q(t_1)} \biggr)
= \frac{\delta t_1}{4} + O(t_1^2).
\end{align*}
The Poincaré inequality, on the other hand,
implies the contraction in $L^2(\mu)$
(see Section~4.2.2 of \cite{BGLMarkov}):
for all $t_2 \in (0,1]$,
\[
\lVert P_{t_2} \rVert_{2 \to 2} \leqslant \exp(-\rho_2 t_2).
\]
Thus, by chaining the two estimates,
\[
\lVert P_{t_1 + t_2} \rVert_{2 \to q(t_1)}
\leqslant \exp \biggl( \frac{\delta t_1}{4} - \rho_2 t_2 + O(t_1^2) \biggr).
\]
We take $t_1 = \rho_2 h$ and $t_2 = \delta h/4$ for some $h$ sufficiently small.
Then we get
\[
\lVert P_{(\rho_2 + \delta /4) h} \rVert_%
{2 \to 2 + 2\rho_1 \rho_2 h + O(h^2)}
\leqslant 1 + O(h^2).
\]
By the reverse argument of Gross
(see again Theorem~5.2.3 of \cite{BGLMarkov}
or, more explicitly, Théorème~2.8.5 of \cite{LSIToulouse}),
this implies that $\mu$ satisfies a log-Sobolev inequality with constant
\[
\frac{\rho_1\rho_2}{\rho_2 + \delta/4}. \qedhere
\]
\end{proof}

Having shown the alternative tightening result,
we can now apply it to obtain that $m^N_*$ satisfies a log-Sobolev inequality
with constant
\[
\biggl( \frac{1}{\rho'}
+ \frac{\delta}{4\rho'(\rho - M^F_{mm}/N)} \biggr)^{\!-1}
\]
Inserting the explicit expressions for $\rho'$ and $\delta$
in the first step concludes. \qed

\section{Proof of Theorem~\ref{thm:concentration}}

We follow the approach in the recent work of Jackson and Zitridis
\cite{JacksonZitridisConcentration} to prove the concentration of measure.
In this method, the most important step
is to establish the following ``controlled'' entropy decay estimate.

\begin{prop}
\label{prop:entropy-decay}
Let $t \geqslant 1$.
Under the setting of Theorem~\ref{thm:concentration},
suppose that $(\mu_s)_{s \in [0,t]}$ is a flow of probability measures solving
\begin{equation}
\label{eq:eom-Y}
\dd Y_s = \bigl( - \nabla U(Y_s) + \alpha(s,Y_s) \bigr) \dd s
+ \sqrt 2 \dd B'_s,\qquad \mu_s = \Law(Y_s)
\end{equation}
for some $\alpha \in \mathcal C ([0,t] \times \R^d ; \R^d)$
with $\nabla\alpha \in L^1\bigl([0,t]; L^\infty(\R^d)\bigr)$,
where $B'_s$ is another Brownian motion.
Then,
\[
H(\mu_t | m_*)
\leqslant \frac{M^2+3}{3}e^{-\rho(t-1)} W_2^2(\mu_0,m_*)
+ \frac 12 \int_0^t\!\!\int_{\R^d}
\lvert \alpha(s,y)\rvert^2 \mu_s(\dd y) \dd s.
\]
\end{prop}

\begin{proof}[Proof of Proposition~\ref{prop:entropy-decay}]
The proof consists of two steps: short-time regularization
from Wasserstein to entropy
and long-time contraction in entropy.

\proofstep{Step 1: Short-time regularization}
In this step, we apply the \emph{coupling by change of measure} method
developed by Guillin, F.-Y.~Wang and P.~Ren in a series of works
\cite{GuillinWangDegenerate,
WangHypercontractivityHamiltonian, RenWangExponential}.
Let $(X^*_s)_{s \in [0,1]}$ be a stationary overdamped Langevin process
\[
\dd X^*_s = - \nabla U(X^*_s) \dd s + \sqrt 2 \dd B_s,
\qquad \Law(X^*_s) = m_*
\]
defined on the probability space $(\Omega, \mathbb F, \mathbb P)$.
Let the initial values of the processes $X^*$, $Y$ be optimally coupled:
\[
\Expect [ \lvert Y_0 - X^*_0 \rvert^2 ] = W_2^2(\mu_0, m_*).
\]
Let us define the process $(Y_s)_{s \in [0,1]}$ by
\[
\dd Y_s = - \nabla U(X^*_s) \dd s - (Y_0 - X^*_0) \dd s + \sqrt 2\dd B_s,
\]
where, as we recall, $B_s$ is the Brownian motion driving $X^*_s$.
This construction ensures that
\[
\dd (Y_s - X^*_s) = - (Y_0 - X^*_0) \dd s.
\]
Thus,
\[
Y_s - X^*_s = (1 - s) (Y_0 - X^*_0)
\]
and in particular, $Y_1 = X^*_1$.
For $s \in [0,1]$, define the stochastic process
\[
\xi_s \coloneqq - \nabla U(X^*_s) - (Y_0 - X^*_0) + \nabla U(Y_s)
- \alpha(s, Y_s).
\]
By the Hessian bound on $U$, its norm can be bounded:
\begin{align*}
\lvert \xi_s \rvert^2
&\leqslant 4 \lvert \nabla U(X^*_s) - \nabla U(Y_s) \rvert^2
+ 2 \lvert \alpha(s,Y_s) \rvert^2
+ 4 \lvert Y_0 - X^*_0 \rvert^2 \\
&\leqslant 4 \bigl( M^2(1-s)^2 + 1 \bigr) \lvert Y_0 - X^*_0\rvert^2
+ 2 \lvert \alpha(s,Y_s) \rvert^2.
\end{align*}
By the definition of $\xi_s$, the process $Y$ satisfies
\[
\dd Y_s = \bigl( - \nabla U(Y_s) + \alpha(s,Y_s) \bigr) \dd s
+ \xi_s \dd s + \sqrt 2 \dd B_s,
\]
and thus, for
\[
B'_s \coloneqq B_s + \frac{\xi_s}{\sqrt 2} \dd s,
\]
the equation of motion \eqref{eq:eom-Y} is formally verified for $Y_s$,
except that $B'_s$ is not a Brownian motion under $\mathbb P$.
Define the exponential local martingale
\[
R_s \coloneqq \exp \biggl( - \int_0^s \frac{\xi_u}{\sqrt 2} \cdot \dd B_u
- \frac 14 \int_0^s \lvert \xi_u\rvert^2 \dd u \biggr).
\]
We can verify that $(R_s)_{s \in [0,1]}$ is a real martingale
by the method of \cite[Appendix~C]{uklpoc}.
By the Girsanov theorem, under the probability $\mathbb Q = R_1 \mathbb P$,
the process $(B'_s)_{s \in [0,1]}$ is a Brownian motion
so the equation of motion \eqref{eq:eom-Y} is really verified
under $\mathbb Q$.

Let $f$ be an arbitrary function from $\R^d$ into $[c,C]$
for some $0 < c < C < \infty$.
We have
\begin{multline*}
\Expect_{\mathbb Q} [\log f(Y_1)]
= \Expect_{\mathbb Q} [\log f(X^*_1)]
= \Expect_{\mathbb P} [R_1\log f(X^*_1)] \\
\leqslant \Expect_{\mathbb P} [R_1 \log R_1]
+ \log \Expect_{\mathbb P}[f(X^*_1)],
\end{multline*}
where the inequality is due to Young's convex duality.
Note that by the bound on $\lvert\xi_t\rvert^2$ above,
\begin{align*}
\Expect_{\mathbb P} [R_1 \log R_1]
&= \frac 14 \int_0^1 \Expect_{\mathbb Q}
[ \lvert \xi_s \rvert^2 ] \dd s \\
&\leqslant \frac {M^2 + 3}3 \Expect [ \lvert Y_0 - X^*_0 \rvert^2 ]
+ \frac 12 \int_0^1 \Expect [ \lvert \alpha(s,Y_s)\rvert^2] \dd s \\
&\leqslant \frac {M^2 + 3}3 W_2^2(\mu_0, m_*)
+ \frac 12 \int_0^1\!\!\int_{\R^d} \lvert \alpha(s,y)\rvert^2
\mu_s(\dd y)\dd s
\end{align*}
Thus, the \emph{log-Harnack inequality} of F.-Y.~Wang holds:
\[
\Expect_{\mathbb Q} [\log f(Y_1)]
\leqslant \log \Expect_{\mathbb P} [f(X^*_1)]
+ \frac {M^2 + 3}3 W_2^2(\mu_0, m_*)
+ \frac 12 \int_0^1\!\!\int_{\R^d} \lvert \alpha(s,y)\rvert^2
\mu_s(\dd y)\dd s.
\]
Taking a sequence of functions $f$ that tends to
the relative density $\dd\mu_1/\!\dd m_1$ in the log-Harnack above, we get
\begin{equation}
\label{eq:entropy-short-time}
H(\mu_1 | m_*) \leqslant \frac{M^2 + 3}{3} W_2^2(\mu_0,m_*)
+ \frac 12 \int_0^1\!\!\int_{\R^d} \lvert \alpha(s,y)\rvert^2
\mu_s(\dd y)\dd s.
\end{equation}
The proposition is proved for the case $t=1$.

\proofstep{Step 2: Long-time contraction}
For the long-time behavior, we write the Fokker--Planck equation
for the flow of measures $(\mu_s)_{s \in [1,t]}$:
\[
\partial_s \mu_s = \Delta \mu_s + \nabla \cdot
\bigl( (\nabla U - \alpha) \mu_s \bigr).
\]
Using the Fokker--Planck equations we find that
the evolution of the relative entropy $H(\mu_t|m_*)$ formally verifies
\begin{align*}
\frac{\dd H(\mu_s|m_*)}{\dd s}
&= \int_{\R^d} \nabla \log \frac{\dd \mu_s}{\dd m_*}
\cdot \biggl( -\nabla \log \frac{\dd \mu_s}{\dd m_*} + \alpha(s,\cdot) \biggr)
\dd \mu_s \\
&= - I(\mu_s | m_*)
+ \int_{\R^d} \nabla \log \frac{\dd \mu_s}{\dd m_*} \cdot \alpha(s,\cdot)
\dd \mu_s \\
&\leqslant - \frac 12 I(\mu_s | m_*)
+ \frac 12 \int_{\R^d} \lvert \alpha(s,y)\rvert^2 \mu_s(\dd y) \\
&\leqslant - \rho H(\mu_s | m_*)
+ \frac 12 \int_{\R^d} \lvert \alpha(s,y)\rvert^2 \mu_s(\dd y),
\end{align*}
where the inequality on the third line is due to Cauchy--Schwarz
and that on the last is due to the $\rho$-log-Sobolev inequality
for $m_*$.
So formally by Grönwall's lemma,
\[
H(\mu_t | m_*) \leqslant
e^{-\rho (t-1)} H(m_1 | m_*)
+ \frac 12 \int_1^t \!\! \int_{\R^d}
\lvert \alpha(s,y) \rvert^2 \mu_s(\dd y) \dd s.
\]
The arguments above can be justified by using
the approximation arguments in \cite[Section~4.2]{uklpoc}.
Inserting the upper bound \eqref{eq:entropy-short-time} on $H(\mu_1|m_*)$,
we get
\[
H(\mu_t | m_*)
\leqslant \frac{M^2+3}{3}e^{-\rho(t-1)} W_2^2(\mu_0,m_*)
+ \frac 12 \int_0^t\!\!\int_{\R^d}
\lvert \alpha(s,y)\rvert^2 \mu_s(\dd y) \dd s,
\]
which is the full claim of the proposition.
\end{proof}

Having shown the entropy decay, we are ready to prove
Theorem~\ref{thm:concentration} by convex duality arguments.
Let $m_t$ be the overdamped Langevin flow
as in the statement of the theorem.
Suppose first that its initial condition is deterministic:
$m_0 = \delta_x$ for some $x \in \R^d$.
Fix $t \geqslant 1$.
Let $f \colon \R^d \to \R$ be an arbitrary $\mathcal C^1_\textnormal{b}$
function with a Lipschitz constant $\leqslant 1$
and let $\lambda > 0$.
According to the variational formulation of relative entropy, we have
\[
\log m_t [e^{\lambda f}]
= \sup_{\mu_{[0,t]} \in \mathcal P(\mathcal C([0,t]; \R^d))}
\bigl( - H(\mu_{[0,t]} | m_{[0,t]})
+ \lambda \mu_t[f] \bigr),
\]
where $\mathcal P\bigl(\mathcal C([0,t]; \R^d)\big)$ denotes
the set of probability measures on the space
of continuous paths on $[0,t]$ that take values in $\R^d$,
and $\mu_t$ is the marginal of path-space measure $\mu_{[0,t]}$
at the terminal time $t$.
Arguing as in the first step of the proof of
Proposition~3.1 of \cite{JacksonZitridisConcentration}%
\footnote{The only difference here is that we work on the whole space
instead of the torus.
For this reason, we have taken a test function $f$
that is bounded with bounded derivatives.
Note that this function corresponds to the terminal condition
in the stochastic control formulation.
In particular, by the Hamilton--Jacobi--Bellman equation,
the optimal ``control variable'' $\alpha$ has the Lipschitz regularity
$\lVert\nabla \alpha(s,\cdot)\rVert_{L^\infty}
= O((t-s)^{-1/2})$.}%
, we find that the optimization problem on the right is equivalent to
\[
\sup_{(\alpha, \mu)}
\biggl(- \frac 14 \int_0^t\!\!\int_{\R^d}
\lvert \alpha(t,y) \rvert^2 \mu_t(\dd y)\dd t + \lambda \mu_t[f] \biggr),
\]
where the supremum is taken over
all $\alpha$ satisfying the condition of
Proposition~\ref{prop:entropy-decay}
and all continuous flow of probability measures $\mu = (\mu_s)_{s \in [0,t]}$
solving weakly the Fokker--Planck equation:
\[
\partial_t \mu = \Delta \mu + \nabla \cdot
\bigl( (\nabla U - \alpha) \mu \bigr).
\]
Thus, by applying Proposition~\ref{prop:entropy-decay},
\begin{align*}
\log m_t[e^{\lambda f}]
&= \sup_{(\alpha,\mu)}
\biggl( -\frac 14\int_0^t\!\!\int_{\R^d}
\lvert \alpha(s,y)\rvert^2 \mu_s(\dd y) \dd s
+ \lambda \mu_t[f] \biggr) \\
&\leqslant \sup_{(\alpha,\mu)} \biggl(
\frac{M^2+3}{6} e^{-\rho(t-1)} W_2^2(\delta_x, m_*)
- \frac 12 H(\mu_t | m_*) + \lambda \mu_t[f] \biggr).
\end{align*}
The log-Sobolev inequality for $m_*$ implies
the T\textsubscript{$\`1$} inequality:
\[
2H(\mu_t | m_*) \geqslant \rho W_1^2(\mu_t, m_*).
\]
So for all $\mu_t \in \mathcal P(\R^d)$,
\begin{align*}
- \frac 12 H(\mu_t | m_*) + \lambda \mu_t[f]
&\leqslant - \frac{\rho}{4} W_1^2(\mu_t, m_*) + \lambda \mu_t[f] \\
&\leqslant - \frac{\rho}{4} W_1^2(\mu_t, m_*) + \lambda W_1(\mu_t, m_*)
+ \lambda m_*[f] \\
&\leqslant \frac{\lambda^2}{\rho} + \lambda m_*[f],
\end{align*}
the second inequality being the Kantorovich duality.
Thus,
\[
\log m_t[e^{\lambda f}]
\leqslant \frac{M^2+3}{6} e^{-\rho(t-1)} W_2^2(\delta_x, m_*)
+ \frac{\lambda^2}{\rho} + m_*[f].
\]
The inequality holds also for all $1$-Lipschitz
(not necessarily $\mathcal C^1_\textnormal{b}$) function $f$
by approximation.
By Markov's inequality,
\begin{align*}
\Proba \bigl[ f(X_t) - m_*[f] \geqslant r \bigr]
&\leqslant e^{-\lambda r} m_t \bigl[ \exp \bigl( \lambda (
f - m_*[f] ) \bigr) \bigr] \\
&\leqslant \exp \biggl(
\frac{M^2+3}{6} e^{-\rho(t-1)} W_2^2(\delta_x, m_*)
+ \frac{\lambda^2}{\rho} - \lambda r \biggr).
\end{align*}
Taking $\lambda = r\rho/2$, we get the desired result
for the case $m_* = \delta_x$.
The general case follows by conditioning on the initial value. \qed

\section{Additional comments}

\subsubsection*{Simpler assumption}

Under a stronger regularity assumption on the energy functional $F$,
the uniform Poincaré condition on $m^{N,i|-i}_*$ (the fourth point)
follows from the uniform log-Sobolev condition on $\hat m$ (the third point).
Indeed, the log-density of $m^{N,i|-i}_*$ satisfies
\[
- \nabla_i \log m^{N,i|-i}_* (x^i | \vect x^{-i})
= D_m F (\mu_{\vect x}, x^i),
\]
while for $\hat \mu_{\vect x^{-i}}$ we have
\[
- \nabla_i \log \hat \mu_{\vect x^{-i}} (x^i)
= D_m F (\mu_{\vect x^{-i}}, x^i).
\]
So if the energy functional verifies
\[
\lvert D_m F (\mu_{\vect x}, x^i)
- D_m F(\mu_{\vect x^{-i}}, x^i) \rvert \leqslant \frac MN
\]
for some constant $M$ (this is true if
$(m,x,x')\mapsto\nabla_x\frac{\delta^2\`F}{\delta m^2}(m,x,x')$
is bounded for example),
then the derivatives can at most differ by $M/N$.
That is to say, for $N$ large,
the conditional measure $m^{N,i|-i}_*(\cdot|\vect x^{-i})$
is a weak log-Lipschitz perturbation
of the hat measure $\hat \mu_{\vect x^{-i}}$.
Since the latter $\hat \mu_{\vect x^{-i}}$ satisfies a $\rho$-log-Sobolev,
\emph{a fortiori}, it satisfies a $\rho$-Poincaré inequality.
We can then apply the perturbation result of Aida and Shigekawa
\cite[Theorem 2.7]{AidaShigekawaLSI} to obtain
the desired uniform Poincaré inequality for $m^{N,i|-i}_*$
(see also the work of Cattiaux and Guillin \cite{CattiauxGuillinFunctional}
for explicit constants).

\subsubsection*{Lipschitz transport method}

Recently another work of Kook, M.\,S.~Zhang, Chewi, Erdogdu and M.~Li
\cite{KZCELSampling}
addresses the uniform log-Sobolev problem under flat convexity
by the Lipschitz transport method
(see in particular Appendix~B of the fourth arXiv version of the article).
In addition to the assumption of this note,
they suppose that the energy functional $F$ decomposes as follows:
\[
F(m) = \int_{\R^d} V(x) m(\dd x) + F_0(m)
\]
for some $V\colon \R^d \to \R$ that is the sum of a strongly convex
and a Lipschitz function,
and some $F_0 \colon \mathcal P_2(\R^d) \to \R$
whose intrinsic derivative $D_m F_0$ is uniformly bounded.
These additional structures are crucially used
to construct a reverse heat flow that sends a standard Gaussian
in $\R^{Nd}$ to the Gibbs measure $m^N_*$.
They estimate the Lipschitz constant of the mapping induced by the reverse flow
following a recent observation of Brigati and Pedrotti
\cite{BrigatiPedrottiHeatFlow}:
since the log-Hessian of the measure along the heat flow is related
to the covariance of exponentially tiltings of $m^N_*$,
it suffices to control the covariance matrices
in order to prove the existence of a Lipschitz transport mapping.
The covariance control is achieved
by generalizing the propagation of chaos result
of our previous joint work \cite{ulpoc}
to the case where the particles do not live in the same environment
but still interact according to the functional $F_0$.
See \cite[Lemma 24]{KZCELSampling}.
The drawback of the transport approach, as remarked in the article,
is that the log-Sobolev constant obtained depends
on the mean field interaction size doubly exponentially.
It seems also to the author that when $N \to \infty$,
the log-Sobolev constant obtained for $m^N_*$ does not tend to
the constant supposed for mean field invariant measure,
thus not reaffirming the conjecture of \cite{DGPSPhase}.
The constant also becomes weaker exponentially when the dimension increases.

\subsubsection*{Other consequences of Theorem~\ref{thm:lsi}}

As mentioned in the introduction, log-Sobolev inequalities
are a key tool to obtain exponential convergence in entropy.
The entropy is the natural quantity
to consider when treating large particle systems,
as it scales linearly when the number of particles increase.
More precisely, typical initial conditions $m^N_0$
of the particle system \eqref{eq:ps} satisfy
\[
H(m^N_0 | m^N_*) = O(N).
\]
Thus applying the entropy contraction yields
\[
H(m^N_t | m^N_*) \leqslant O\bigl(N \exp (- 2\rho^N t)\bigr).
\]
However, since the log-Sobolev constant $\rho^N$ obtained
in Theorem~\ref{thm:lsi}
is always weaker than the mean field constant $\rho$,
this is not the optimal result that we can obtain.
Indeed, for $t$ such that $H(m^N_t | m^N_*)$ is much larger than $1$,
we can apply the defective contraction in \cite{ulpoc}
(which follows directly from the defective log-Sobolev) to get
\[
H(m^N_t| m^N_*) \leqslant O\Bigl(N \exp\bigl( -2(1-\varepsilon)\rho t\bigr)
+ 1\Bigr),
\]
and then for $t$ large,
we chain it with the perfect (as opposed to defective) contraction.
The $N$-uniform log-Sobolev inequality implies also
a time-uniform propagation of chaos property
by the approach of \cite[Theorem 3.7]{DGPSPhase},
but the result will be weaker than that of \cite{ulpoc}
due to the weaker rate $\rho^N < \rho$.

For the kinetic case, as the tensor product
\[
m^N_*(\dd \vect x) \otimes \frac{\exp(- \lvert \vect v\rvert^2\!/2)\dd\vect v}
{(2\pi)^{Nd/2}}
\]
satisfies a $\min(\rho^N,1)$-log-Sobolev inequality,
we can apply the linear entropic hypo\-coercivity results of Villani
\cite[Part I]{VillaniHypocoercivity}.
Note that the hypocoercive construction is dimension-free, as demonstrated
in \cite[Lemma 4.9]{uklpoc}, so is suitable for studying large particle systems.
Similarly, we can use the triangle argument to show
time-uniform propagation of chaos, as done in \cite{GuillinMonmarcheKinetic},
but the result will be again weaker than that of \cite{uklpoc}
as $\rho^N < \rho$.

\subsubsection*{Optimal constants for log-Sobolev}

The result of Theorem~\ref{thm:lsi} provides only a lower bound
(or an upper bound for different conventions)
on the best log-Sobolev constant for the Gibbs measure $m^N_*$.
Therefore, although our result does not reaffirm
the conjecture posed in \cite{DGPSPhase}, it does not provide
a counter-example to it either.
Here we only mention that the size of the defect $\delta$
in the inequality
\[
I(m^N | m^N_*) \geqslant
2\rho' \bigl( \mathcal F^N(m^N) - N\mathcal F(m_*) \bigr) - \delta
\]
is $O(1)$ when $N \to \infty$
and is optimal according to the Gaussian example in \cite{ulpoc}.
Thus one cannot find a sequence of log-Sobolev constants
$\rho^N$ with $\lim_{N\to\infty}\rho^N = \rho$
by the tightening approach.

\subsubsection*{Long-time concentration of measure}

The long-time concentration of measure property
of the particle system \eqref{eq:ps} is essential
to justifying and quantifying the validity of Monte Carlo algorithms.
Early efforts of Malrieu \cite{MalrieuLSI}
and of Bolley, Guillin and Villani \cite{BGVConcentration}
focus on the two-body interaction case with convex potentials.
A recent work of W.~Liu, L.~Wu and C.~Zhang \cite{LWZLongTime}
addresses the case of possibly non-convex confinement potentials
with small interactions.
Their approach consists of coupling particle systems
with reflected Brownian motions constructed component-by-component.
The coupling by reflection leads to
a Wasserstein contraction, or in other words, a Lipschitz spectral gap,
crucially used to demonstrate the desired long-time concentration
(see proof of Proposition~5.3 therein).
As remarked by the authors in the article,
the Lipschitz spectral gap obtained via coupling
implies an $L^2$ spectral gap, or equivalently, a Poincaré inequality,
which is one of the objectives of their previous joint work with Guillin
\cite{GLWZUPLSI}.
In this note,
we exploit the contraction in entropy related to the gradient flow structure,
instead of the Wasserstein contraction above,
and obtain a lower bound on the large deviation rate functional
for the diffusion process.
The lower bound then imply a non-asymptotic upper bound
on the deviation probability
according to the observation of Jackson and Zitridis
\cite[Proposition~3.1]{JacksonZitridisConcentration}.
This approach towards concentration inequalities is not entirely new:
we can interpret the result of Bobkov and Götze
\cite[Theorem~1.3]{BobkovGoetzeExponential}
as such an implication for the i.\,i.\,d.\ case,
where the relative entropy is the large deviation functional by Sanov's theorem.

\printbibliography

\end{document}